\documentclass[12pt]{article}
\usepackage{amsthm,amsmath,amssymb,amsfonts}
\usepackage{epsfig}
\usepackage{psfrag}
\usepackage{enumerate}
\usepackage{url}
\usepackage{fullpage}
\usepackage{fancyhdr}
\usepackage{latexsym}
\usepackage{color}
\usepackage{hyperref}
\usepackage[nottoc,numbib]{tocbibind}
\usepackage{cite}

\newtheorem{theorem}{Theorem}[section]
\newtheorem{definition}[theorem]{Definition}
\newtheorem{lemma}[theorem]{Lemma}
\newtheorem{corollary}[theorem]{Corollary}

\newtheorem{proposition}[theorem]{Proposition}

\newtheorem{example}[theorem]{Example}
\newtheorem{remark}[theorem]{Remark}

\def\hpic #1 #2 {\mbox{$\begin{array}[c]{l} 
\epsfig{file=#1,height=#2}\end{array}$}}
\def\wpic #1 #2 {\mbox{$\begin{array}[c]{l} 
\epsfig{file=#1,width=#2}\end{array}$}}

\def\C{\mathbb C}

\def\N{\mathbb N}

\def\R{\mathbb R}

\def\CA{{\cal {A}}}
\def\CB{{\cal {B}}}

\def\CE{{\cal {E}}}

\def\CP{{\cal {P}}}

\def\be{\begin{equation}}
\def\ee{\end{equation}}
\def\bt{\begin{Theorem}}
\def\et{\end{Theorem}}
\def\bi{\begin{itemize}}
\def\ei{\end{itemize}}
\def\bea{\begin{eqnarray}}
\def\eea{\end{eqnarray}}
\def\beast{\begin{eqnarray*}}
\def\eeast{\end{eqnarray*}}
\def\ben{\begin{enumerate}}
\def\een{\end{enumerate}}

\def\bi{\bibitem}

\newcommand{\e }{\mbox{$ \epsilon $}}

\def\rar{\rightarrow}
\def\Rar{\Rightarrow}

\def\Lra{{\Leftrightarrow}}

\def\e{{\epsilon}}

\hypersetup{colorlinks=true,citecolor=blue}

\numberwithin{equation}{section}

\date{}

\begin{document}
\title{Continuous minmax theorems}

%\noindent                                 
\author{Madhushree Basu\footnote{Indian Statistical Institute, Bangalore; \small email: madhushree@isibang.ac.in}
\\ V.S. Sunder\footnote{The Institute of Mathematical
  Sciences, Chennai; \small e-mail: sunder@imsc.res.in}}

\maketitle%    General info                                                     

\begin{abstract}
  In classical matrix theory, there exist useful extremal
  characterizations of eigenvalues and their sums for Hermitian
  matrices (due to Ky Fan, Courant-Fischer-Weyl and Wielandt) and some
  consequences such as the majorization assertion in Lidskii's
  theorem. In this paper, we extend these results to the context of
  self adjoint elements of finite von Neumann algebras, and their
  distribution and quantile functions.  This work was motivated by a
  lemma in \cite{BerVoiUnbdd} that described such an extremal
  characterization of the distribution of a self-adjoint operator
  affiliated to a finite von Neumann algebra - suggesting a possible
  analogue of the classical Courant-Fischer-Weyl minmax theorem, for a
  self adjoint operator in a finite von Neumann algebra.
  It is to be noted that the
    only von Neumann algebras considered here have separable
    pre-duals.
\end{abstract}

\section{Introduction}
This paper is arranged as follows: in Section 2,
we prove an extension of the classical minmax theorem of Ky Fan's (\cite{kyfan})
in a von-Neumann algebraic setting for self adjoint operators having no atoms in their distributions,
and then, as an application of the above, in Section 3,
we give a few applications of the previous section. First
we state and prove an exact analogue of the Courant-Fischer-Weyl minmax theorem
(\cite{CH}) for operators inside $II_1$ factors having no eigenvalues. It is interesting to note
that classically Courant-Fischer-Weyl minmax theorem came before Ky Fan's theorem
for Hermitian matrices whereas the order of events
is reversed in our proofs. Then, as an application of our version
of Courant-Fischer-Weyl minmax theorem, we prove that that if a self adjoint operator
with no eigenvalues, is dominated by another such operator (both being inside a $II_1$ factor),
then their respective quantile functions are dominated one by the other in the same order.
Finally we discuss a
continuous analogue of Lidskii's theorem - a majorization-type inequality
between eigenvalues of sum of Hermitian matrices
and the sum of eigenvalues of the summand matrices,
discussions and proofs of the finite dimensional
version of which can be found in \cite{lid1}, \cite{lid2}, \cite{wie}.
In Section 4, we state and prove, for operators with no eigenvalues
in $II_1$ factors,
a continuous analogue of Wielandt's minmax theorem (\cite{wie}), the classical
version of which gives an extremal characterization of arbitrary sums of eigenvalues of Hermitian matrices.

Similar continuous analogues of minmax-type results
have been worked out earlier, for example in \cite{kf} and \cite{hiai}, in the general context
of defining trace-measurablity of operators affiliated to von Neumann algebras
and for generalizing the concept of majorization in von Neumann algebras. However
in these papers, the emphasis have been on positive operators and the rigorous proofs
for the von Neumann algebraic adaptation of
minmax-type results, corresponded to singular values of Hermitian matrices.
On the other hand, our proofs are simple, independent of the approach of
these papers, deal explicitly with self adjoint (as against positive)
operators in certain von Neumann algebras and correspond to
eigenvalues (as against singular values) of Hermitian matrices in the finite dimensional case.
Moreover as far as we know, unlike former works on this topic,
our formulations, for the particular case of finite dimensional matrix algebras,
give the exact statements of the classical Ky Fan's, Courant-Fischer-Weyl's and
Wielandt's theorems. However in the continuous case,
our results are restricted by non-atomicity property of the distribution
of the self adjoint operator under consideration as well as by
certain properties of the underlying von Neumann algebra.

In order to describe our results, which are continuous analogues of
certain inequalities that appear as part of the set of inequalities
mentioned in Horn's conjecture (\cite{horn}), it will be convenient to
re-prove the well-known fact that any monotonic function with appropriate
one-sided continuity is the distribution function of a random variable $X$
- which can in fact be assumed to be defined on the familiar Lebesgue
space $[0,1)$ equipped with the Borel $\sigma$-algebra and Lebesgue
measure. (We adopt the convention of \cite{BerVoiUnbdd} that the \textit{distribution function} $F_\mu$
of a compactly supported probability measure\footnote{Actually Bercovici and Voiculescu considered possibly
unbounded self-adjoint operators
affiliated to $M$, so as to also be able to handle probability measures which are not necessarily compactly supported,
but we shall be content with the case of bounded $a \in M$, having a compactly supported
probability measure as its distribution.}  $\mu$ defined on the $\sigma$-algebra $\CB_\R$ of Borel  sets in
$\R$, is {\em left-continuous}; thus $F_\mu(x) = \mu((-\infty,x)$).

\begin{proposition}\label{linv}
If $F:\R \rar [0,1]$ is monotonically non-decreasing and left
continuous and if there exists $\alpha,\beta \in \R$ with $\alpha<\beta$ such that
%\be \label{lts} \lim_{t\rar -\infty} F(t) = 0 \mbox{ and } \lim_{t\rar +\infty} F(t) = 1~,\ee
\be \label{lts} F(t)=0, \text{ for } t \le \alpha \mbox{ and } F(t) = 1 \text{ for } t \ge \beta,\ee
then there exists a monotonically non-decreasing right-continuous function
$X:[0,1) \rar \R$ such that $F$ is the 
distribution function of $X$, i.e., $F(t) = m(\{s: X(s) < t\})$,
where $m$ denotes the Lebesgue measure on $[0,1)$. Moreover $range(X) \subset [\alpha, \beta]$.
\end{proposition}

\begin{proof}Define $X:[0,1) \rar \R$ by
\begin{align}\label{Xdef}
X(s) &= \inf\{t: F(t) > s\}\\
&= \inf\{t: t \in E_s\}~, \nonumber
\end{align}
where $E_s = \{t \in \R: F(t) > s\} ~\forall s \in [0,1)$. (The hypothesis \ref{lts}
is needed to ensure that $E_s$ is a non-empty bounded set for every $s \in [0,1)$ so that, indeed $X(s) \in \R$.)

 First deduce from the monotonicity of $F$ that 
\begin{align*}
s_1 \leq s_2 &\Rar E_{s_2} \subset E_{s_1}\\
&\Rar X(s_1) \leq X(s_2)
\end{align*}
and hence $X$ is indeed monotonically non-decreasing. 

The definition of $X$ and the fact that $F$ is monotonically non-increasing and left
continuous are easily seen to imply that $E_s = (X(s),\infty)$, and hence, it is seen that
\begin{align}\label{X*m1} 
X(s) < t ~ &\Lra ~ \exists t_0 < t \mbox{ such that } F(t_0) > s \nonumber \\
& \Lra ~ F(t) >  s \text{ (since $F$ is left-continuous)}
\end{align}

Hence, if $t \in \R$
\begin{align}\label{X*m2}
m(\{s \in [0,1): X(s) < t\}) = m([0, F(t)) = F(t), \text{ proving the required statement.}
\end{align}

Moreover,  if for any $s \in [0,1)$, $X(s)< \alpha$, then by definition of $X$, $\exists~t'< \alpha$
such that $F(t') > s \ge 0$, a contradiction to the first hypothesis in \ref{lts}. On the other hand, if 
for any $s \in [0,1)$, $X(s) > \beta$, then
by \ref{X*m1}, $s \ge F(\beta)=1$ (by the second hypothesis in \ref{lts}), a contradiction. Hence indeed $range(X) \subset [\alpha, \beta]$.  
\end{proof}

This function $X$ is known as \textit{quantile function}\footnote{This function acts as the inverse of
the distribution function at every point that is not an atom of the probability measure $\mu$.} of the distribution $F$. If $F=F_\mu$ for
a probability measure $\mu$ on $\R$, then $X$ is denoted as $X_\mu$. $X$ can also
be thought of as an element of $L^{\infty}(\R, \mu)$, where $\mu$ is a compactly supported probability
measure on $\R$ such that $\mu=m \circ X^{-1}$ and $supp~\mu \subset [\alpha,\beta]$. We will elaborate
on this later in Proposition \ref{pro}.

\bigskip

Given a self-adjoint element $a$ in a von Neumann algebra $M$ and a (usually faithful normal) tracial state $\tau$ on $M$,
define 

\begin{align} \label{dist} \mu_a(E) := \tau(1_E(a))\end{align}
(for the associated scalar spectral measure) to be the \textit{distribution of} $a$.
Since $\tau$ is positivity preserving, $\mu_a$ indeed turns out to be a probability measure on $\R$.

For simplicity we write $F_a, X_a$ instead of $F_{\mu_a}, X_{\mu_a}$
(to be pedantic, one should also indicate the dependence on $(M,\tau)$,
but the trace $\tau$  and the $M$ containing $a$ will usually be clear.)
Note that only the abelian von Neumann subalgebra $A$ generated by $a$ and $\tau|_A$
are relevant for the definition of $F_a$ and $X_a$.

\bigskip 
For $M, a, \tau$ as above, it was shown in \cite{BerVoiUnbdd} that
\be \label{bvminmax}
1 - F_{\mu_a}(t) = \max\{\tau(p): p \in \CP(M), pap \geq ta\}.
\ee

\begin{example}\label{mneg}
Let  $M = M_n(\C)$ with $\tau$ as the tracial state on this $M$. If $a=a^* \in M$
has distinct eigenvalues $\lambda_1 <\lambda_2 < \cdots < \lambda_n$, then
$F_{a}(t) = \frac{1}{n} |\{j: \lambda_j < t\}| = \sum_{j=1}^n \frac{j}{n} 1_{(\lambda_j, \lambda_{j+1}]}$.
We see that the distinct numbers less than 1 in the range of $F_a$ are attained at the $n$ distinct eigenvalues of $a$,
and further that equation \ref{bvminmax} for $t = \lambda_j$ says that $n-j+1$
is the largest possible dimension of a subspace $W$ of $\C^n$ such that $\langle a\xi, \xi \rangle \geq \lambda_j$
for every unit vector $\xi \in W$. In other words equation \ref{bvminmax} suggests
a possible extension of the classical Courant-Fischer minmax theorem for a self adjoint operator in a von Neumann algebra, involving
its distribution.

It is also true and not hard to see that the  right side of  equation \ref{bvminmax} is indeed a maximum (and not just a supremum),
and is in fact attained at a spectral projection of $a$; i.e., the two sides of
equation \ref{bvminmax} are also equal to $\max\{\tau(p): p \in \CP(A), pap \geq ta\}$, where $A$
is the abelian von Neumann subalgebra generated by $a$.
\end{example}

\section{Our version of Ky Fan's theorem}
In this section we wish to proceed towards obtaining non-commutative counterparts of the classical Ky Fan's minmax theorem
formulated for appropriate self-adjoint elements of appropriate finite von Neumann algebras.

\begin{proposition}\label{pro}
Let $(\Omega,\CB,P)$ be a probability measure, and suppose $Y:\Omega \rar \R$ is an essentially bounded random variable.
Let $\sigma(Y) = \{Y^{-1}(E): E \in \CB_\R\}$ and let  $\mu = P\circ Y^{-1}$ be the distribution of $Y$.
Then, for any $s_0 \in F_\mu(\R)$, we have
\begin{align} \label{ines}
&\inf\{\int_{\Omega_0} Y dP: \Omega_0 \in \sigma(Y),P(\Omega_0) \ge s_0\} \nonumber\\
&= \inf\{\int_E f_0 d\mu: E \in \CB_\R,\mu(E) \ge s_0\} \nonumber\\
&= \inf\{\int_G X_\mu dm: G \in \sigma(X_\mu),m(G) \ge s_0\} \nonumber\\
&= \int_0^{s_0} X_\mu dm,
\end{align}
where $f_0 = id_\R$ and $m$ denotes Lebesgue measure on [0,1).
\end{proposition}

\begin{proof} The version of the change of variable theorem we need says that if $(\Omega_i,\CB_i,P_i), i=1,2$
are probability spaces and $T:\Omega_1 \rar \Omega_2$ is a measurable function such that
$P_2 = P_1 \circ T^{-1}$, then
\be \label{cov}
\int_{\Omega_2} g dP_2 = \int_{\Omega_1} g \circ T dP_1~,
\ee
 for every bounded measurable function $g : \Omega_2 \rar \R$.
 
For every $\Omega_0 \in \sigma(Y)$, which is of the form $Y^{-1}(E)$ for some $E \in \CB_\R$,
set $G = X_\mu^{-1}(E)$. Notice, from equations \ref{X*m1} and \ref{X*m2} that 
\begin{align*}
m \circ X_\mu^{-1}(-\infty,t) &= \mu(\{s \in [0,1): X_\mu(s) < t\})\\
&= \mu(\{s \in [0,1): s < F_\mu(t)\})\\
&= F_\mu(t)\\
&= \mu(-\infty,t)~;
\end{align*}
i.e. $m \circ X_\mu^{-1} = \mu = P \circ Y^{-1}$. Now, set
$g = 1_E \cdot f_0$. Since $g \circ Y = 1_E \circ Y \cdot Y = 1_{Y^{-1}(E)} Y = 1_{\Omega_0} Y$,
and (similarly) $g \circ X_\mu = 1_G X_\mu$, we see that the first two equalities in \ref{ines}
are immediate consequences of two applications of the version stated in equation \ref{cov} above, of the `change of variable' theorem.

As for the last, if $G \in \CB_{[0,1)}$ with $m(G) \ge s_0$, then write
$I = G \cap [0,s_0), J = [0,s_0) \setminus I, K = G  \setminus I$ and note that $G = I \coprod K, [0,s_0) = I \coprod J$
(where $\coprod$ denotes disjoint union, and $K = G \setminus [0,1) \subset [s_0,1)$.
 So we may deduce that
 \begin{align*}
 \int_G X_\mu dm - \int_0^{s_0} X_\mu dm &=  \int_K X_\mu dm - \int_J X_\mu dm\\
 &\geq X_\mu(s_0) m(K) - X_\mu(s_0) m(J)\\
  &\geq 0~,
 \end{align*}
 since $s_1 \in J, s_2 \in K \Rar s_1 \leq s_0 \leq s_2 \Rar X_\mu(s_1) \leq X_\mu(s_0) \leq X_\mu(s_2)$ (by the monotonicity of $X_\mu$),
 and $m(K) \ge m(J)$. 
 Thus, we see that
 \[\inf\{\int_G X_{\mu} dm: G \in \sigma(X_{\mu} ), m(G) \ge s_0\} \geq \int_0^{s_0} X_\mu dm ~, \] 
 while conversely, 
  \[\inf\{\int_G X_{\mu}  dm: G \in \CB_\R, m(G) \ge s_0\} \leq \int_{[0,s_0)} X_{\mu} dm = \int_0^{s_0} X_\mu dm ~, \] 
  thereby establishing the last equality in \ref{ines}. 
 \end{proof}

\begin{theorem}\label{kf}
Let $a$ be a self-adjoint element of a von Neumann algebra
 $M$ equipped with a faithful normal tracial state $\tau$. Let $A$ be the von Neumann subalgebra generated by $a$ in $M$
 and $\CP(M)$ be the set of projections in $M$.
 Then, for all $s \in F_a(\R)$,
 \begin{align} \label{m=a}
&\inf\{\tau(ap): p \in \CP(M), \tau(p) \ge s\} \nonumber\\
&=\inf\{\tau(ap): p \in \CP(A), \tau(p) \ge s\} \nonumber\\
&=  \int_0^s X_a dm
\end{align}
(hence the infima are attained and are actually minima), if either:
\ben
\item (`continuous case') $\mu_a$ has no atoms,  or
\item (`finite case') $M=M_n(\C)$ for some $n \in \N$ and $a$ has spectrum $\{\lambda_1 < \lambda_2 < \cdots < \lambda_n\}$.
\een
\end{theorem}

\begin{proof} 
We begin by noting that in both the cases, the last equality in \ref{m=a} is an immediate consequence of Proposition \ref{pro}.
Moreover the set
$\{\tau(ap): p \in \CP(A), \tau(p) \ge s\}$ being contained in $\{\tau(ap): p \in \CP(M), \tau(p) \ge s\}$, it is clear that
\[\inf\{\tau(ap): p \in \CP(A), \tau(p) \ge s\}  \ge \inf\{\tau(ap): p \in \CP(M), \tau(p) \ge s\}.\]

So we just need to prove that
\begin{align}\label{reqd}
\inf\{\tau(ap): p \in \CP(A), \tau(p) \ge s\}  \le \inf\{\tau(ap): p \in \CP(M), \tau(p) \ge s\}.
\end{align}

\ben
\item (the continuous case) Due to the assumption of $\mu_a$ being compactly supported and having no atoms,
it is clear that $F_a$ is continuous and that $F_a(\R) = [0,1]$. 

Under the standing assumption of separability of pre-duals of our von
Neumann algebras, the hypothesis of this case implies the existence of
a probability space $(\Omega, \CB, P)$ and a map $\pi:A \rar
L^\infty(\Omega, \CB, P)$ such that $\int \pi(x) dP = \tau(x) ~\forall
x \in A$, $Y:=\pi(a)$ is a random variable and $\pi$ is an isomorphism
onto $L^\infty(\Omega, \sigma(Y), P)$.

We shall  establish the first equality of \ref{m=a} by showing that if $p' \in \CP(M)$ and $\tau(p') = s$,
then $\tau(ap') \geq \min\{\tau(ap): p \in \CP(A), \tau(p) \ge s\}$. For this, first note that since $\tau$ is a faithful normal
tracial state on $M$, there exists a $\tau$-preserving conditional expectation $\CE:M \rar A$. Then
\begin{align*}
\tau(ap') = \tau(a\CE(p'))= \int YZ dP,
\end{align*}
where $Z = \pi(\CE(p'))$.  Since $\CE$ is linear and positive, it is clear that $0 \leq Z \leq 1$ $P-a.e.$ So it is enough to prove that
\begin{align*}
&\inf\{\int_\Omega YZ dP: 0 \leq Z \leq 1, \int Z dP \ge s\}\\
&= \inf\{\int_E Y dP: E \in \CB, P(E) \ge s\}.
\end{align*}
For this, it is enough, thanks to the Krein-Milman theorem (see, e.g. \cite{KM}), to note that $K = \{Z \in L^\infty(\Omega,\CB,P): 0 \leq Z \leq 1,\int Z dP \ge s\}$
is a convex set which is compact in the weak* topology inherited from $L^1(\Omega,\CB,P)$, and prove that the set
$\partial_e(K)$ of its extreme points is $\{1_E: P(E) \ge s\}$.

For this, suppose $Z \in K$ is not a projection, Clearly then $ P(\{Z \in (0,1)\}) > 0$, so there exists $\e > 0$
such that $P(\{\e < Z < 1-\e\})>0$. Since $\mu_a$, and hence $P$
has no atoms, we may find disjoint Borel subsets
$E_1, E_2 \subset \{Z \in (\e, 1-\e)\}$ such that $P(E_1) = P(E_2) > 0$. If we now set $Z_1 = Z + \e(1_{E_1} - 1_{E_2})$
and $Z_2 = Z + \e(1_{E_2} - 1_{E_1})$,
it is not hard to see that $ Z_1, Z_2 \in K, Z_1 \neq Z_2$ and $Z = \frac{1}{2}(Z_1 + Z_2) $ showing that $Z \notin \partial_e(K)$ ,
thereby proving \ref{reqd}.

\bigskip

\item (the finite case) Since $a$ has distinct eigenvalues $\lambda_1 < \lambda_2  < \cdots < \lambda_n$,
$A$ is a maximal abelian self-adjoint subalgebra of $M_n(\C)$.  Recall that in this case,
$F_{a}(t) = \frac{1}{n} |\{j: \lambda_j < t\}| = \sum_{j=1}^n \frac{j}{n} 1_{(\lambda_j, \lambda_{j+1}]}$.
It then follows that $F_a(\R) = \{\frac{j}{n}: 0 \leq j \leq n\}$ and that $X_a = \sum_{j=1}^n \lambda_j 1_{[\frac{j-1}{n},\frac{j}{n})}$
and \ref{m=a} is then (after multiplying by $n$) precisely the statement of Ky Fan's theorem
(in the case of self-adjoint matrices with distinct eigenvalues):

For $1 \leq j \leq n,$

\begin{align*}
&\inf\{\tau(ap): p \in \CP(M_n(\C)), rank(p) \ge j\}  \\
&= \inf\{\tau(ap): p \in \CP(A), rank(p) \ge j\} = \frac{1}{n}\sum_{i=1}^j \lambda_i=\int_0^{\frac{j}{n}} X_a(s) ds. 
\end{align*}

It suffices to prove the following:

\begin{align*}
\inf\{\tau(ap): p \in \CP(A), rank(p) \ge j\} \le \inf\{\tau(ap): p \in \CP(M_n(\C)), rank(p) \ge j\}. 
\end{align*}

For this, begin by deducing from the compactness of $\CP(M_n(\C))$ that there exists a $p_0 \in \CP(M_n(\C))$
with $rank(p_0) \ge j$ such that $\tau(ap_0) \leq \tau(ap) ~\forall p \in \CP(M_n(\C)) ~\mbox{with }rank(p) \ge j$.
We assert that any such minimizing $p_0$ must belong to $A$. The assumption that $A$ is a masa means we only need to prove
that $p_0a=ap_0$. For this pick any self-adjoint $x \in M_n(\C)$, and consider the function $f:\R \rar \R$ defined by $f(t) = \tau(e^{itx}p_0e^{-itx}a)$.
Since clearly $e^{itx}p_0e^{-itx} \in \CP(M)$ and $rank(e^{itx}p_0e^{-itx}) = rank(p_0) \ge j$,  for all $t \in \R$,
we find that $f(t) \geq f(0) ~\forall t$. As $f$ is clearly differentiable, we may conclude that $f^\prime(0) = 0$. Hence,
\[0 = \tau(ixp_0a - ip_0xa)= i (\tau(xp_0a) - \tau(p_0xa))= i (\tau(xp_0a) - \tau(xap_0)),\]
 so that $\tau(x(p_0a-ap_0)) = 0$ for all $x=x^* \in M$, and indeed $ap_0=p_0a$ as desired.
\een
\end{proof}

Case 1 of Theorem \ref{kf} is our continuous formulation of Ky Fan's result
while Case 2 only captures the classical Ky Fan's theorem for the case of distinct eigenvalues.
However the general case of non-distinct eigenvalues can also be deduced from our proof,
as we show in the following corollary:

\begin{corollary}\label{kfcor}
Let $a$ be a Hermitian matrix in $M_n(\C)$ with spectrum $\{\lambda_1 \le \cdots \le \lambda_n\}$,
where not all $\lambda_j$s are necessarily distinct. Then for all $j \in \{1,\cdots, n\}$,
\[\min\{\tau(ap): p \in \CP(M_n(\C)), ~rank(p) \ge j\} =\frac{1}{n} \sum_{i=1}^j \lambda_i. \]
\end{corollary}

\begin{proof}
We may assume that $a$ is diagonal. Let $A_1$ be the set of all diagonal matrices,
so that $A \subsetneq A_1$. Pick $a^{(m)} = diag(\lambda^{(m)}_1, \lambda^{(m)}_2, \cdots , \lambda^{(m)}_n) \in A_1$
such that $\lambda^{(m)}_j$s are all distinct and $\lim_{m \rar \infty} \lambda^{(m)}_j = \lambda_j ~\forall 1 \leq j \leq n$.
Then the already established case of Theorem \ref{kf} in the case of distinct eigenvalues shows that for all $p \in \CP(M_n(\C))$ with $rank(p) \ge j$, we have
\begin{align*}
\tau(ap) &= \lim_{m \rar \infty} \tau(a^{(m)}p)\\
&\geq \lim_{m \rar \infty}\frac{1}{n}\sum_{i=1}^j \lambda^{(m)}_i\\
&= \frac{1}{n}\sum_{i=1}^j \lambda_i
%\\ &= \min\{tr(ap): p \in \CP(M_n(\C)), ~rank(p) = j\},
\end{align*}

The above, along with the fact that $\tau(ap_j) = \frac{1}{n}\sum_{i=1}^j \lambda_i$,
where $p_j$ is the obvious diagonal projection,
completes our proof of Ky Fan's theorem for Hermitian matrices in full generality.
\end{proof}

\begin{remark}\label{eq}
 It is not difficult to see that equation \ref{m=a} holds even if we replace the inequality $\tau(p) \ge s$ with equality.
\end{remark}

\begin{remark}\label{rem}
Notice that the hypothesis and hence the conclusion, of the `continuous case' of Theorem \ref{kf} are satisfied by any
self-adjoint generator of a masa in a $II_1$ factor.
\end{remark}

\section{Applications of our version of Ky Fan's theorem}

In this section we discuss three applications of our version of Ky Fan's theorem, but we first identify a
necessary definition for many of our `continuous cases' , as well as a lemma which we will need to use at a later stage.

\begin{definition}\label{ctscase}
Given a self-adjoint element $a$ in a finite von Neumann algebra $M$, we say that we
are in the {\bf continuous case} if for  $B \in \{M,A\}$ (with $A$ the von Neumann subalgebra generated by $a$ in $M$)  and $p \in
\CP(B)$, we have $\{\tau(r): r \in \CP(B), r \leq p\} = [0,\tau( p]$.
 (This assumption for $B = A$ amounts to requiring that $\mu_a$ has no atoms)
 \end{definition}

\begin{lemma}\label{hallem}
With $M,a,A$ in the continuous case as above, suppose $t_0 < t_1 \in \R, F(t_1)-F(t_0) = \delta > 0$   and let $r_0= 1_{[t_0, \infty)}(a)$ and $q_0=1_{[t_0, t_1)}(a)$.

Then $r_0, q_0 \in \CP(A), \tau(r_0)=1-F(t_0), \tau(q_0) = \delta$ and $q_0 \le r_0$ and
\[ \tau(aq_0)=\min_{\substack{q \in \CP(A)\\q \le r_0 \\ \tau(q) = \delta}} \tau(aq) = \min_{\substack{q \in \CP(M)\\q \le r_0 \\ \tau(q) = \delta}} \tau(aq)\]
\end{lemma}

\begin{proof}
If we consider any other $q \in A,$ with $ q \le r_0 $ and $\tau(q) = \delta$, then $q$ is of the form $1_E(a)$, where $E \subset [t_0,\infty)$ with $\mu_a(E) = \delta$. Arguing as in the proof of Proposition \ref{pro},

\begin{align*}& \int_{t_0}^{ t_1} t ~ d\mu(t) \le \int_E t ~ d\mu(t) \\
\Rightarrow &\int_{F(t_0)}^{F(t_1)}X(s)~ds \le \int_{F(E)} X(s)~ds \\
\Rightarrow & \tau(aq_0) \le \tau(aq).
\end{align*}

To prove the same for any $q \le r_0$, first we note that since $r_0 \in W^*(\{a\})$,
$(M_0, \tau_0):=\Big(r_0M r_0, \frac{\tau(\cdot)}{\tau(r_0)}\Big)$ is also a von Neumann
algebra with equipped with a faithful normal tracial state and $a_0:=r_0ar_0$ is
a self adjoint element with a continuous distribution $\mu_0$ (with respect to $\tau_0$) in it. 

Let the von Neumann subalgebra generated by $a_0$ in $M_0$ be $A_0$.
Then $M_0$ and $A_0$ satisfy the same `continuity hypotheses as $M$ and $A$.

Any $q \le r_0$ with $\tau(q)= \delta$ can be thought of as $q \in \CP(M_0)$ with $\tau_0(q) = \frac{\delta}{\tau(r_0)}$,
and conversely.

Now as in the proof of the continuous case of Theorem \ref{kf} we can assume that
there exists a non-atomic probability space $(\Omega_0, \CB_0, P_0)$
and a map $\pi_0:A_0 \rar L^\infty(\Omega_0, \CB_0, P_0)$
such that $\int \pi_0(x) dP_0 = \tau_0(x) ~\forall x \in A_0$, $Y_0:=\pi_0(a_0)$
and $\pi_0$ is an isomorphism onto $L^\infty(\Omega_0, \sigma(Y_0), P_0)$.  \\

We proceed exactly as we did in the proof of the `continuous case' of Theorem \ref{kf} to show that
$\text{min}\{ \int Y_0Z_0 ~dP_0: Z_0 \in L^\infty(\Omega_0, \CB_0, P_0), 0 \le Z_0 \le 1, \int Z_0 dP_0 = \frac{\delta}{\tau(r_0)}\}$
is indeed attained and the minimizing contractions can only be of the form $1_E \in L^\infty(\Omega_0, \CB_0, P_0)$
for $E \in \sigma(Y_0), P_0(E) = \frac{\delta}{\tau(r_0)}$.
 
 Thus we have
\begin{align*}
\nonumber &\tau_0(a_0q_0) = \min_{\substack{q \in \CP(M_0) \\ \tau_0(q) = \delta_0}} ~ \tau_0(a_0q) \\
\nonumber \Rightarrow &\frac{\tau(a_0q_0)}{\tau(r_0)}=\min_{\substack{q \in \CP(M)\\ q \le r_0 \\ \frac{\tau(q)}{\tau(r_0)}= \frac{\delta}{\tau(r_0)}}} \frac{\tau(a_0q)}{\tau(r_0)} \\%
\nonumber \Rightarrow & \tau(aq_0)=\min_{\substack{q \in \CP(M)\\q \le r_0 \\ \tau(q) = \delta}} \tau(aq),
\text{ since $r_0$ commutes with $a$ and any $q\le r_0$.}
\end{align*}
\end{proof}

\begin{remark}\label{lemhal}
By considering $-a$ in place of $a$, for instance, we clearly have the following dual to Lemma \ref{hallem}:

With $M,a,A,t_0, t_1$ as in Lemma \ref{hallem}, let
$p:=1_{(-\infty, t_1)}(a)$ and $\tilde{q}:=1_{[t_0, t_1)}(a) \le p$. Then, $p, \tilde{q} \in \CP(A), \tau( p)= F(t_1), \tau(\tilde{q}) = \delta$, and
\[ \tau(a\tilde{q})=\max_{\substack{q \in \CP(A)\\q \le p \\ \tau(q) = \delta}} \tau(aq) = \max_{\substack{q \in \CP(M)\\q \le p\\ \tau(q) = \delta}} \tau(aq)\]
\end{remark}

We now proceed to our generalization of
the classical Courant Fischer-Weyl minmax theorem:

\begin{theorem} \label{cfthm}
Let $a$ be a self adjoint element
of a von Neumann algebra $M$ equipped with a faithful normal tracial state $\tau$.
Let $t_0$ and $t_1 \in \R$ such that $t_0 < t_1$ and $F_a(t_1)-F_a(t_0)=:\delta >0$. Then

\begin{align}\label{cf} 
\int_{F_a(t_0)}^{F_a(t_1)} X_a(s) ~ ds = \sup_{\substack{r \in \CP(M) \\ \tau(r) \ge 1-F_a(t_0)}} ~
\inf_{\substack{q \in \CP(M) \\ q \le r \\ \tau(q) = \delta}} ~ \tau(aq),
\end{align}

if either
\begin{enumerate}
 \item we are in the `continuous case'; or 
 \item  (`finite case') $M$ is a type $I_n$ factor for some $n \in \N$
   and $a$ has spectrum $\{\lambda_1 < \lambda_2 < \cdots < \lambda_n\}$.
 
 Moreover there exists $r_0 \in \CP(A) \le \CP(M)$ with $\tau(r_0) \ge 1-F(t_0)$ such that 
\[\int_{F_a(t_0)}^{F_a(t_1)} X_a(s) ~ ds =\min_{\substack{q \in \CP(M) \\ q \le r_0 \\ \tau(q) = \delta}} ~ \tau(aq),\]
so that the supremum is actually maximum. 
\end{enumerate}

\end{theorem}

\begin{proof}
For simplicity we write $F$ and $X$ for $F_a$ and $X_a$ respectively. 
\begin{enumerate}
 \item (the continuous case) For proving ``$\le$'', deduce, from
 Lemma \ref{hallem} that
\be \label{cf} \int_{F(t_0)}^{F(t_1)} X(s) ~ ds  
\le \sup_{\substack{r \in \CP(M) \\ \tau(r) \ge 1-F(t_0)}} ~ \inf_{\substack{q \in \CP(M)\\ q \le r \\ \tau(q) = \delta}} ~ \tau(aq).
\ee

\bigskip

For ``$\ge$'', let us choose any projection $r$ with $\tau(r) \ge 1-F(t_0)$. 

Let $r_1=1_{(-\infty, t_1)}(a)$. Then $\tau(r_1)=F(t_1) \Rightarrow \tau(r_1 \wedge r) \ge F(t_1)-F(t_0)=\delta$. 

Hence, by the hypothesis in this continuous case, $\exists~q_1 \le r \wedge r_1$ with $\tau(q_1) = \delta$. 

Now consider the $II_1$ factor $(M_1, \tau_1) := \Big(r_1 M r_1, \frac{\tau(\cdot)}{\tau(r_1)}\Big)$,
where $\tau_1$ is a faithful normal tracial state on $M_1$. Then $q_1$ can be thought
of as a projection in $\CP(M_1)$ with $\tau_1(q_1) = \frac{\delta}{\tau(r_1)}$.

Note that $q_0=1_{[t_0,t_1)}(a) \le r_1$.

As above $a_1:=r_1ar_1$ is a self adjoint element with continuous distribution in $M_1$.
So we can consider our version of Ky Fan's theorem in $M_1$ (Theorem \ref{kf}) (also see Remark \ref{eq}):

\[\frac{\int_0^{F(t_0)} X(s) ds}{\tau(r_1)}=\tau_1(a(r_1-q_0))=\min_{\substack{q \in \CP(M_1) \\ \tau_{1}(q) = \frac{F(t_0)}{\tau(r_1)}}} \tau_1(aq).
\]

(using the fact that $a, q_0$ and $q\in \CP(M_1)$ commute with $r_1$.)

Subtracting both sides from $\tau_1(a_1)$ and writing  $q'$ for $r_1-q$ in the index, we can rewrite it as:

\[\frac{\int_{F(t_0)}^{F(t_1)} X(s)~ ds}{\tau(r_1)}=\max_{\substack{q' \in \CP(M_1) \\ \tau_1(q')= \frac{F(t_1)-F(t_0)}{\tau(r_1)} = \frac{\delta}{\tau(r_1)}}} \tau_1(aq'),\]

or equivalently,

\[\int_{F(t_0)}^{F(t_1)} X(s)~ ds=\max_{\substack{q' \in \CP(M)\\ q' \le r_1 \\ \tau(q') = \delta}} \tau(aq').\]

Now using the fact that $q_1 \le r \wedge r_1$,
we have:

\[ \int_{F(t_0)}^{F(t_1)} X(s)~ ds=\max_{\substack{q' \in \CP(M)\\ q' \le r_1 \\ \tau(q') = \delta}} \tau(aq) \ge \tau(aq_1)
\ge \inf_{\substack{q \in \CP(M)\\ q \le r \\ \tau(q) = \delta}} \tau(aq),\]

thus, and using the fact that our choice of $r$ was arbitrary with $\tau(r) \ge 1-F(t_0)$, we have:

\begin{align}\label{**}
 \int_{F(t_0)}^{F(t_1)} X(s)~ ds \ge \sup_{\substack{r \in \CP(M) \\ \tau(r) \ge 1-F(t_0)}} ~ \inf_{\substack{q \in \CP(M)\\ q \le r \\ \tau(q) = \delta}} ~ \tau(aq). 
\end{align}

Equations \ref{cf} and \ref{**} together give us the required equality.

\bigskip

\item (the finite case) Notice that if we set $t_0=\lambda_i, t_1=\lambda_{i+j}, \delta=\frac{j}{n}$,
 where $i,j\in \{1,\cdots, n\}$ such that $i+j-1 \le n$, equation \ref{cf} translates to:

\[\lambda_i+\lambda_{i+1} + \cdots + \lambda_{i+j-1} = \sup_{\substack{r \in \CP(M_n(\C)) \\ Tr(r)\ge n-i+1}}~
\inf_{\substack{q \in \CP(M_n(\C)) \\ q \le r\\ Tr(q)= j}} Tr(aq),\]

where $Tr$ is the sum of the diagonal entries of matrices. 

For the inequality ``$\le$'' we prove,

\[\lambda_i + \lambda_{i+1} + \cdots + \lambda_{i+j-1}=Tr(aq_0)=\min_{\substack{q \in \CP(M_n(\C)) \\q \le r_0\\ Tr(q)= j} } Tr(aq),\]

where $r_0= 1_{\{\lambda_i,\lambda_{i+1},\cdots,\lambda_{n}\}}(a)$ and $q_0=1_{\{\lambda_i, \lambda_{i+1},\cdots,\lambda_{i+j-1}\}}(a)$,

by first showing that any minimizing projection below $r_0$ has to commute with $r_0ar_0$,
and then using the fact that with distinct eigenvalues $r_0ar_0$ generates a masa in $r_0M_n(\C)r_0$, concluding that minimizing projections
have to be spectral projections (see the exactly similar proof of the finite case of Theorem \ref{kf}).

%The case of non-distinct eigenvalues follows as a limiting case of the above, as we saw in the proof of Theorem \ref{kf}. 

\bigskip

For proving ``$\ge$'', we start with an arbitrary projection $r$ with $Tr(r)\ge n-i+1$ and note that if we define $r_1:=1_{\{\lambda_1,\cdots,\lambda_{i+j-1}\}}(a)$,
then $\exists~ q_1 \le r \wedge r_1$ such that $Tr(q_1) = j$. Now we proceed
using Ky Fan's theorem for finite dimensional Hermitian matrix $r_1ar_1$ in $r_1 M_n(\C) r_1$,
exactly as in the above proof of the continuous case of this theorem.
\end{enumerate}
\end{proof}

\begin{remark}
Theorem \ref{cfthm} can equivalently be stated as:
 
\[\int_{F(t_0)}^{F(t_1)} X(s) ~ ds = \inf_{\substack{p \in \CP(M) \\ \tau(p) \ge F(t_1)}} ~
\sup_{\substack{q \le p \\ \tau(q) = \delta}} ~ \tau(aq),\]

Moreover we can get the classical Courant-Fischer-Weyl minmax theorem for Hermitian matrices in full generality (i.e. involving non-distinct eigenvalues as well)
from the above theorem in exactly similar manner as in Corollary \ref{kfcor}.
\end{remark}

The classical Courant-Fischer-Weyl minmax theorem has a natural corollary
that says if $a,b$ are Hermitian matrices in $M_n(\C)$ such that
$a \le b$ (i.e. $b-a$ is positive semi-definite),
and if $\{\alpha_1 \le \cdots \le \alpha_n\}$ and $\{\beta_1 \le \cdots \le \beta_n\}$
are their spectra respectively, then $\alpha_j \le \beta_j$ for all $j \in \{1,\cdots,n\}$.
As expected, Theorem \ref{cfthm} leads us to the same corollary for the `continuous case':
 
\begin{corollary}\label{cfcor}
 Let $M$ be a $II_1$ factor equipped with faithful normal tracial state $\tau$. If $a,b \in M$ such that $a=a^*,b=b^*$
and $\mu_a,\mu_b$ have no atoms. Then
\begin{align}\label{leq}
a \le b \Rightarrow  X_a \le X_b.
\end{align}
\end{corollary}

\begin{proof}
Notice that since $a \le b$ and $\tau$ is positivity preserving, we have
\begin{align}\label{tau}
 \tau(xax^*) \le \tau(xbx^*).
\end{align}
for all $x \in M$.

Fix $0 \leq s_0 < s_1 < 1$.

By our assumptions on $a$ and $b$,
$\mu_a, \mu_b$ are compactly supported probability measures with no atoms.
Hence $F_a$ and $F_b$ are continuous functions with $range(F_a)=range(F_b)= [0,1]$.
Thus $\exists~ t_0^a,t_1^a, t_0^b, t_1^b \in \R$ such that $s_0=F_a(t_0^a)=F_b(t_0^b)$ and $s_1=F_a(t_1^a)=F_b(t_1^b)$.

Now using Theorem \ref{cfthm}
\begin{align*}
 \int_{s_0}^{s_1} X_a~dm &= \sup_{\substack{r \in \CP(M) \\ \tau(r) \ge 1-F_a(t_0^a)}} \inf_{\substack{q \in \CP(M) \\ q \le r \\ \tau(r)=s_1-s_0}} \tau (aq) \\
                         & = \sup_{\substack{r \in \CP(M) \\ \tau(r) \ge 1-F_a(t_0^a)}} \inf_{\substack{q \in \CP(M) \\ q \le r \\ \tau(r)=s_1-s_0}} \tau (qaq) \\
                         & \le \sup_{\substack{r \in \CP(M) \\ \tau(r)
                             \ge 1-F_b(t_0^b)}} \inf_{\substack{q \in
                             \CP(M) \\ q \le r \\ \tau(r)=s_1-s_0}}
                         \tau (qbq), ~\text{by the inequality \ref{tau}}\\
                         &= \sup_{\substack{r \in \CP(M) \\ \tau(r) \ge 1-F_b(t_0^b)}} \inf_{\substack{q \in \CP(M) \\ q \le r \\ \tau(r)=s_1-s_0}} \tau (bq) \\
                         &= \int_{s_0}^{s_1} X_b~dm.
\end{align*}

This proves that 
\begin{align}\label{int}
 \int_I X_a ~dm \le \int_I X_b ~dm
\end{align}
for any interval $I=[s_0,s_1) \subset [0,1)$, and in fact for any $I
\in \CA:=\{\sqcup_{j=1}^k [s_0^j,s_1^j): 0 \le s_0^j < s_1^j <1, k \in \N\}$. 

But $\CA$ is an algebra of sets which generates the $\sigma$-algebra
$\CB_{[0,1)}$. 
Thus for any Borel $E \subset [0,1)$, there exists a sequence $\{I_n:n
\in \N\} \subset \CA$ such that $\mu(I_n \Delta E) \rightarrow 0$.

Recall from Proposition \ref{linv} that our quantile functions of self
adjoint elements of von Neumann algebras  
are elements of $L^{\infty} ([0,1),\CB_{[0,1)}, m)$. We may hence
deduce from the sentence following equation. (\ref{int})that if $E, I_n$
are the previous paragraph, we have:
\begin{align*}
\int_E X_a dm &= \lim_{n \rightarrow \infty} \int_{I_n} X_a dm\\
&\le \lim_{n \rightarrow \infty} \int_{I_n} X_a dm\\
&= \int_E X_b dm.
\end{align*}
As $E \in \CB_{[0,1)}$ was arbitrary, this shows that, $X_a \le
X_B m-a.e.$; as $X_a, X_b$ are continuous by our hypotheses, this
shows that indeed $X_a \le X_b$.
\end{proof}

\bigskip
Finally, we discuss a continuous analogue of Lidskii's majorization result.

By Theorem \ref{kf}, we have the following lemma:

\begin{lemma}\label{lid1}
If $M$ is a von Neumann algebra with a faithful normal tracial state $\tau$ on it,
then for $a=a^*, b=b^* \in M$ with $\mu_a, \mu_b$ non-atomic and for all $s \in [0,1)$, 
\[\int_0^{s} X_{a+b} ~dm  \ge \int_0^{s} (X_{a} +X_{b}) ~dm.\]
Moreover,
\[\int_0^{1} X_{a+b} ~dm  = \int_0^{1}( X_{a}+ X_{b}) ~dm.\]
\end{lemma}

\begin{proof}
Recall from our proof of Theorem \ref{kf} that there exists a
projection $p \in \CP(M)$ (in fact in the von Neumann algebra
generated by $a+b$) such that $\tau(p) \ge s$ and
\begin{align*}
\int_0^{s} X_{a+b} ~dm &= \tau((a+b)p) \\
&=\tau(ap)+\tau(bp) \\
&\ge \inf\{\tau(ap'): p' \in \CP(M), \tau(p') \ge s\}+\inf\{\tau(bp'): p' \in \CP(M), \tau(p') \ge s\} \\
&=\int_0^{s} X_{a} ~dm+\int_0^{s} X_{b} ~dm \\
&=\int_0^s (X_a +X_b)~dm.
\end{align*}

Finally, it is clear (from our change-of-variable argument in
Proposition \ref{pro} for instance) that for any $c=c^* \in M$, we
have $\int_0^1 X_c dm = \tau(c)$ and hence
\[\int_0^1 X_{a+b}~dm = \tau(a+b)=\tau(a)+\tau(b)=\int_0^1 X_a~dm+\int_0^1 X_b~dm=\int_0^1 (X_a +X_b)~dm.\]
\end{proof}

The above is an analogue of the fact that for $n \times n$ Hermitian matrices $a, b$,
with their eigenvalues $\lambda_1 (a) \le \cdots \le \lambda_n (a)$ and
$\lambda_1 (b) \le \cdots \le \lambda_n (b)$, for all $k \in \{1,\cdots,n-1\}$,
\[ \sum_{j=1}^k \lambda_j(a+b) \ge \sum_{j=1}^k \lambda_j(a)+\sum_{j=1}^k \lambda_j(b),\]
and 
\[ \sum_{j=1}^n \lambda_j(a+b) = \sum_{j=1}^n \lambda_j(a)+\sum_{j=1}^n \lambda_j(b),\]
i.e. $\lambda(a)+\lambda(b)$ is majorized by $\lambda(a+b)$ in the sense
of \cite{HLP}.

We consider the definition of majorization in the continuous context (see for example, \cite{sak}) as follows:

\begin{definition}
 For $a=a^*, b=b^*$ in a von Neumann algebra $M$ with a faithful normal tracial state $\tau$
 on it, $a$ is said to be \textit{majorized} by $b$ if $\int_0^s X_a ~dm \le \int_0^s X_b~ dm$
 for all $s \in [0,1)$ and $\int_0^1 X_a ~dm = \int_0^1 X_b~ dm$. When
 this happens, we simply write $X_a \prec X_b$.
\end{definition}

Then, Lemma \ref{lid1} can be written as:

\[X_{a+b} \prec X_a + X_b,\]
which gives a version of the continuous analogue of Lidskii's theorem.

\bigskip

 The study of majorization
and its von Neumann algebraic analogue is vast (see for example,
\cite{kamei1}, \cite{hiai}) and
closely related to the minmax-type results but
we will not discuss it further in this paper.

\section{Continuous version of Wielandt's minmax principle}

In this section we state and prove a continuous analogue of Wielandt's minmax theorem. The classical matrix formulation
of Wielandt's theorem is obtained by taking $\delta = \frac{1}{n}$ when $M = M_n(\C)$,
but we shall not repeat the kind of reasoning given in the case of the Courant-Fischer-Weyl theorem
in the finite-dimensional case where our assumptions of our `continuous case'  are not valid.
We shall be content with formulating and proving the continuous case.

We make the standing `continuity assumption' of Definition \ref{ctscase} throughout this section. Thus our results are
valid for any von Neumann algebra that admits a faithful normal tracial state and has the above-mentioned property.

Our version of Wielandt's theorem is as follows:

\begin{theorem}\label{wiel} 
Let $F, X$ be the distribution and quantile function of $a$.
Let $\delta_j \in \R_+$ and $t_0^j, t_1^j$, $j=1,\cdots,k$, be points in the spectrum of $a$ such that
$t_0^1 < t_1^1 \le t_0^2 < t_1^2 \le \cdots \le t_0^{k-1} < t_1^{k-1} \le t_0^k$ and $F(t_1^j)-F(t_0^j)=\delta_j$, for all $j$. Then

\[\sum_{j=1}^k \int_{[F(t_0^j),F(t_1^j))} X(s) ~ ds =
\inf_{\substack{p_j \in \CP(M) \\ p_1 \le \cdots \le p_k \\ \tau(p_j) \ge F(t_1^j)}} ~
\sup_{\substack{\hat{q}_j \in \CP(M) \\ \hat{q}_j \le p_j \\ \tau(\hat{q}_j) = \delta_j \\ \hat{q}_j \bot \hat{q}_i \text{ for } j \neq i}} ~ \sum_{j=1}^k\tau(a\hat{q}_j).\]
 
 Moreover,  $\exists p_1 \le \cdots \le p_k$ with $p_j \in \CP(A) \subset \CP(M)$, for which there exist mutually orthogonal projections
 $\hat{q}_j \le p_j,\tau(\hat{q}_j)=\delta_j~, \forall j$ such that  
 \[\sum_{j=1}^k \int_{[F(t_0^j),F(t_1^j))} X(s) ~ ds=
 \max_{\substack{\hat{q}_j \le p_j \\ \tau(\hat{q}_j) = \delta_j \\ \hat{q}_j \bot \hat{q}_i}} ~ \sum_{j=1}^k\tau(a\hat{q}_j);\]

\end{theorem}

The following lemmas lead to the proof of the theorem above:

\begin{lemma}\label{l4}
Let $(M,\tau)$ be as above. 
Consider, for any $k \ge 2$, 
\begin{align*} &\{r_1,r_2,\cdots r_k; q'_1, \cdots, q'_{k-1}\} \subset \CP(M),\\ 
&r_1 \geq \cdots \geq r_k,\\
&\tau(r_j) \geq \delta_k + \cdots + \delta_j ~\forall ~1 \leq j \leq k,\\
&q'_j \leq r_j ~\forall 1 \leq j \le k-1,\\
&q'_s q'_t = 0 ~\forall 1 \leq s < t \leq k-1,\\
&\tau(q'_j) =  \delta_j ~\forall 1 \leq j < k-1. 
\end{align*} 
Then there exist mutually orthogonal projections $q_j \leq r_j ~\forall 1 \leq j \leq k$ in $M$,
such that $\sum_{j=1}^k q_j \geq  \sum_{j=1}^{k-1} q'_j$,
and $\tau(q_j) = \delta_j ~\forall 1 \leq j \leq k$.
\end{lemma}

\begin{proof} The proof follows by induction. For $k=2$, choose $q_2 \le r_2$ such that $\tau(q_2)=\delta_2$.

Let $e=q_2 \vee q'_1$.

Then $\tau(e) \le \tau(q_2)+\tau(q'_1) = \delta_2 +\delta_1$ and $e \le r_1$.

But by the hypothesis for $k=2$, $\tau(r_1) \ge \delta_2+\delta_1$.

Hence by the `standing continuity assumption', there exists $f \in
\CP(M)$ such that $e \le f \le r_1$ and $\tau(f)=\delta_2+\delta_1$. 
In particular $q_2 \le e \le f$; thus $f-q_2 \in \CP(M)$ with trace $\delta_1$.

Choose $q_1=f-q_2$. Then $q_j \le r_j$ with trace $\delta_j$ for $j=1,2$
and $q_1+q_2 =f \ge e \ge q'_1$, as required.

\bigskip
Suppose now, for the inductive step, that this result holds with $k$
replaced by $k-1$, and that $r_1, \cdots ,r_k, q_1, \cdots , q_{k-1}$
are as in the statement of the Lemma.

By induction hypothesis - applied to $\{r_2,\cdots,r_k; q'_2, \cdots, q'_{k-1}\}
\subset \CP(M)$ - there exist mutually orthogonal projections
$q_2,\cdots,q_k$ in $M$ such that $q_j \le r_j$ and $\tau(q_j)=\delta_j, \forall 2 \le j \le k$
and 

\begin{align}\label{equation}
\sum_{j=2}^k q_j \ge \sum_{j=2}^{k-1} q'_j
\end{align}.

Let $e_2=q_2+\cdots+q_k$ and $e=e_2 \vee q'_1$.

Then $\tau(e) \le \tau(e_2)+\tau(q'_1) = (\delta_k+\cdots + \delta_2)+\delta_1$ and $e \le r_1$.

But $\tau(r_1) \ge  \delta_k+\cdots + \delta_1$; thus (by the `standing continuity assumption') there exists $f \in \CP(M)$ such that $e \le f \le r_1$
and $\tau(f)=\delta_k+\cdots + \delta_1$.
In particular $e_2 \le e \le f$; thus $f-e_2 \in \CP(M)$ with trace $\delta_1$.

Choose $q_1=f-e_2$.  Then $q_1 \le r_1$ and $q_1 \perp q_j$ for $2 \le j \le k$.

Moreover,
\begin{align*}
q_1+q_2+\cdots+q_k =f  \ge e  &= e_2 \vee q'_1 \\&= \sum_{j=2}^k q_j \vee q'_1 \\
& \ge \sum_{2}^{k-1}q'_j  \vee q'_1 ~\text{ by equation \ref{equation}} \\
&=\sum_{1}^{k-1}q'_j, \end{align*}
thus completing the proof of the inductive step. 
\end{proof}

%\begin{remark}\label{r2}
Lemma \ref{l4} can be rewritten as:   

%Then there exists projection $q \le r_1- \sum_{j=1}^{k-1} q'_j$ and
%mutually orthogonal projections $q_i \leq r_i ~\forall 1 \leq i \leq k$
%such that $\tau(q_i) = \delta ~\forall 1 \leq i \leq k$ and $q+\sum_{j=1}^{k-1} q'_j = \sum_{i=1}^k q_i$.  
%\end{remark}
                                
%\bigskip 

\begin{lemma}\label{lemma}
  Let $(M,\tau)$ be as above. Suppose $\delta_j \in \R_+$, and $\{r_1
  \ge \cdots \ge r_k\} \subset \CP(M)$ such that $\tau(r_j) \ge
  \delta_k + \cdots \delta_j, \forall j=1,\cdots,k$ and suppose we are
  given $(k-1)$ mutually orthogonal projections $q_j'$ such
  that $q_j' \le r_j$ and $\tau(q'_j)=\delta_j~ \forall~j=1,\cdots,k-1$.  Let
\[ e'=q'_1 + \cdots + q'_{k-1} \le r_1.\]
Then there exist projections
$q \le r_1- e', q_j \le r_j~ \forall~j=1,\cdots,k$, such that $\tau(q)=\delta_k$ and $\tau(q_j)=\delta_j~ \forall~j$,
$\{q_j:1 \le j \le k\}$ pairwise mutually orthogonal
and
\begin{align*}\label{lemma1}
 q + e'=q_1 + \cdots + q_k,
\end{align*}
which is also a projection below $r_1$.
\end{lemma}

\begin{proof}
Use Lemma \ref{l4} and choose $q=(q_1+\cdots+q_k)-e'$.
\end{proof}

Before proceeding further, we state a short but useful result:

\begin{lemma}\label{l1}
For $(M, \tau)$ as above and $r, e \in \CP(M)$,
\[\tau(r \wedge e^\perp) \geq \tau( r)- \tau(e)\]
where, of course, $e^\perp = 1-e$.
\end{lemma}

\begin{proof}
\begin{align*}
1 + \tau(r \wedge e^\perp) &\geq \tau( r \vee e^\perp) + \tau(r \wedge e^\perp)\\&= \tau( r) + 1 - \tau(e)
\end{align*}
as required.
\end{proof}

The above results lead to the following lemma:

\begin{lemma}\label{lemma2}
 Let $(M, \tau)$, $t_0^j, t_1^j, \delta_j$ be as in Wielandt's theorem. Let 
$\{r_1 \ge \cdots \ge r_k\}$ and $\{p_1 \le \cdots \le p_k\}$ be sets of projections in $M$ such that
$\tau(p_j) \ge F(t_1^j)$, $\tau(r_j) \ge 1-F(t_0^j)$ for all $1\le j\le k$.
Then there exist mutually orthogonal projections $q_j \le r_j$ and mutually orthogonal
projections $\tilde{q}_j \le p_j$ such that $\tau(q_j)=\tau(\tilde{q}_j)=\delta_j~\forall~j$
and $q_1 + \cdots + q_k = \tilde{q}_1 + \cdots + \tilde{q}_k$.
\end{lemma}

\begin{proof}
 The proof is by induction.
 
 For $k=1$, deduce from Lemma \ref{l1} that
 \begin{align*}
  \tau(p_1 \wedge r_1) &\ge \tau(p_1) - \tau(r_1^{\bot}) \\
                       &\ge \tau(p_1) -1 + \tau(r_1) \\
                       &\ge F(t_1^1) -1 +1-F(t_0^1)\\
                       &=F(t_1^1)-F(t_0^1)\\
                       &=\delta_1,
 \end{align*}
 and thus (by our standing `continuity assumption) there exists a projection $q_1=\tilde{q}_1 \le p_1 \wedge r_1$ of trace $\delta_1$.
 
 For the inductive step, assume $p_1 \le \cdots \le p_k, r_1 \ge
 \cdots r_k$ are as in the lemma and that the lemma is valid with $k$
 replaced by $k-1$. By the induction hypothesis applied to  $p_1 \le
 \cdots \le p_{k-1}, r_1 \ge \cdots \ge r_{k-1}$, there are mutually
 orthogonal projections $q_{j}'\le r_j$ and mutually 
orthogonal projections $\tilde{q}_j \le p_j$ such that
$\tau(q_j')=\tau(\tilde{q}_j)=\delta_j$ for all $j=1,\cdots,k-1$
and $\sum_{j=1}^{k-1} q'_j=\sum_{j=1}^{k-1} \tilde{q}_j=:e'$, say.

Then $e' \le p_{k-1} \le p_k$.

Let $\ell_j = r_j \wedge p_k, ~\forall~j=1,\cdots,k$.% Then $\ell_1 = r_1 \wedge p_k \ge r_1 \wedge p_{k-1} \ge e$.

Then $\ell_k \le \cdots \le \ell_1$. An application of Lemma \ref{l1}, as seen above in the $k=1$ case, gives:

\begin{align*}
 \tau(\ell_j) &\ge F(t_1^k)-F(t_0^j) \\
                           & \ge F(t_1^k)-F(t_0^k)+F(t_1^{k-1})-F(t_0^{k-1})+\cdots + F(t_1^j)-F(t_0^j) \\
                           & = \delta_k +\cdots + \delta_j~ \forall~j=1,\cdots,k.
\end{align*}

Now by Lemma \ref{lemma} - applied with $\ell_j$ in place of $r_j$ -
we may conclude that $\exists~q \le \ell_1 - e', q_j \le \ell_j ~(\le
r_j)$ with $\tau(q)=\delta_k$, 
$\tau(q_j)=\delta_j ~\forall ~j$ and
$q_j \bot q_i ~\forall~j \neq i$, such that $q + e' = q_1 + \cdots + q_k$.

But $q + e'=q + \tilde{q}_1 + \cdots + \tilde{q}_{k-1}$, where $\tilde{q}_{j} \le p_j~\forall~j=1,\cdots,k-1$
and $q \le \ell_1 - e' \le \ell_1=r_1 \wedge p_k$.

Choosing $\tilde{q}_k=q$, the proof of the inductive step is complete. 

\end{proof}

Now we are ready to prove Theorem \ref{wiel}.

\begin{proof}
 For ``$\ge$'' : we take
 $p_j:=1_{(-\infty, t_1^j)}(a)$ and $\tilde{q}_j:=1_{[t_0^j,
   t_1^j)}(a) \le p_j$, and deduce from Remark \ref{lemhal} that
   for each $1 \le j \le k$, we do have
 \beast \int_{F(t_0^j)  }^{F(t_0^j)} X_a dm &=& \tau(a\tilde{q}_j)\\
 &=& \max_{\substack{q \in \CP(M)\\q \le p_j\\ \tau(q) = \delta}} \tau(aq)~.
\eeast
 As the   $\tilde{q}_j. 1 \le j \le k$ are mutually perpendicular, we see that 
 \beast
 \sum_{j=1}^k  \int_{F(t_0^j)  }^{F(t_0^j)} X_a dm &=&  \sum_{j=1}^k \max_{\substack{q_j \in \CP(M)\\q_j \le p_j\\ \tau(q_j) = \delta}} \tau(aq_j)\\
 &\ge& \sup_{\substack{q_j \in \CP(M)\\q_j \le p_j\\ \tau(q_j) = \delta\\q_i \perp q_j \mbox{ for }i \neq j}} \sum_{j=}^k \tau(aq_j)
 \eeast
 and in particular,
 \bigskip
 \[\sum_{j=1}^k \int_{[F(t_0^j),F(t_1^j))} X(s) ~ ds \ge
\inf_{\substack{p_j \in \CP(M) \\ p_1 \le \cdots \le p_k \\ \tau(p_j) \ge F(t_1^j)}} ~
\sup_{\substack{\hat{q}_j \in \CP(M) \\ \hat{q}_j \le p_j \\ \tau(\hat{q}_j) = \delta_j \\ \hat{q}_j \bot \hat{q}_i \text{ for } j \neq i}} ~ \sum_{j=1}^k\tau(a\hat{q}_j).\]

For proving ``$\le$'' here,
let us choose any $p_1 \le \cdots \le p_k$ such that $p_j \in \CP(M)$ and $\tau(p_j) \ge F(t_1^j)$.
%We need to find some $\tilde{q} \le p_k$ such that
%$\tau(\tilde{q}) = k\delta$, $\tau(p_j \wedge \tilde{q}) \ge j\delta$ and $\tau(a\tilde{q}) \le \sum_{j=1}^k \int_{[F(t_0^j),F(t_1^j))} X(s) ~ ds $.

Let $r_j=1_{[t_0^j,\infty)}(a)~\forall~j=1,\cdots,k$. Then
$r_1 \ge \cdots \ge r_k$ with $\tau(r_j)=1-F(t_0^j)$.

Now by Lemma \ref{lemma2}, there exist mutually orthogonal projections $q_j \le r_j$ and
mutually orthogonal projections $\tilde{q}_j \le p_j$ with $\tau(q_j)=\tau(\tilde{q}_j) = \delta_j$
such that $q_1 + \cdots + q_k=\tilde{q}_1 + \cdots + \tilde{q}_k$.

Notice that by our version of Ky Fan's theorem,

\[\tau(aq_j) \ge \inf_{\substack{q \in \CP(M) \\ q \le r_j \\ \tau(q)=\delta_j}} \tau(aq)=\int_{F(t_0^j)}^{F(t_1^j)} X(s)~ ds.\]

Hence,
\[\sum_{j=1}^k \int_{F(t_0^j)}^{F(t_1^j)} X(s)~ ds  \le \sum_{j=1}^k \tau(aq_j)=\sum_{j=1}^k \tau(a\tilde{q}_j)\]

(since $q_1 + \cdots + q_k=\tilde{q}_1 + \cdots + \tilde{q}_k$),
where $\tilde{q}_j \in \CP(M), \tilde{q}_j \le p_j$ with $\tau(\tilde{q}_j)=\delta_j$ and $\tilde{q}_j \bot \tilde{q}_i$.

Hence,
\[\sum_{j=1}^k \int_{[F(t_0^j),F(t_1^j))} X(s) ~ ds \le
\sup_{\substack{\hat{q}_j \in \CP(M) \\ \hat{q}_j \le p_j \\ \tau(\hat{q}_j) = \delta_j \\ \hat{q}_j \bot \hat{q}_i \text{ for } j \neq i}} ~ \sum_{j=1}^k\tau(a\hat{q}_j).\]

As the $p_1 \le \cdots \le p_k$ were chosen arbitrarily, the proof of the theorem is complete.

%\textbf{ Remark:} Since Lemma \ref{lemma} has been proved only for the commutative case, this proves:

%\[\sum_{j=1}^k \int_{[F(t_0^j),F(t_1^j))} X(s) ~ ds =
%\min_{\substack{p_j \in \CP(A) \\ p_1 \le \cdots \le p_k \\ \tau(p_j) \ge F(t_1^j)}} ~
%\sup_{\substack{\hat{q_j} \in \CP(A) \\ \hat{q}_j \le p_j \\ \tau(\hat{q}_j) = \delta \\ \hat{q}_j \bot \hat{q}_i \text{ for } j \neq i}} ~ \sum_{j=1}^k\tau(a\hat{q}_j).\]

\end{proof}

%\begin{remark}
 %For $\delta_1=\cdots = \delta_k=\delta$, the theorem can be written as:

%\[\sum_{j=1}^k \int_{[F(t_0^j),F(t_1^j))} X(s) ~ ds =
%\min_{\substack{p_j \in \CP(M) \\ p_1 \le \cdots \le p_k \\ \tau(p_j) \ge F(t_1^j)}} ~
%\sup_{\substack{q \in \CP(M) \\ q \le p_k \\ \tau(q) = k\delta \\ \tau(q \wedge p_j) \ge j\delta}} ~ \tau(aq).\]
%\end{remark}

\section*{Acknowledgment}

It is a pleasure to record our appreciation of the very readable book
\cite{bhatia} by Rajendra Bhatia, whose proof of the matrix case of Wielandt's theorem we
could assimilate and adapt to the continuous case.
We also wish to thank Manjunath Krishnapur and Vijay
Kodiyalam for helpful discussions.

\bibliographystyle{amsalpha}
\bibliography{references}

\end{document}